\documentclass[11pt]{article}
\ifdefined\pdfinfoomitdate\pdfinfoomitdate=1\fi
\usepackage[margin=1in]{geometry}
\usepackage[T1]{fontenc}
\usepackage{lmodern}
\usepackage{amsmath,amssymb,amsthm,mathtools}
\usepackage{microtype}
\usepackage[colorlinks=true,linkcolor=blue,citecolor=blue,urlcolor=blue]{hyperref}
\hypersetup{pdftitle={Derivative and multiplicity loci of elementary symmetric polynomials in positive characteristic},pdfauthor={Ying Xie},pdfsubject={Exact dimensions, components, and generic multiplicities},pdfcreator={LaTeX}}
\newtheorem{theorem}{Theorem}[section]
\newtheorem{proposition}[theorem]{Proposition}
\newtheorem{lemma}[theorem]{Lemma}
\newtheorem{corollary}[theorem]{Corollary}
\theoremstyle{remark}

\newcommand{\A}{\mathbb A}
\newcommand{\Ga}{\mathbb G_a}

\newcommand{\ind}{\mathbf 1}
\newcommand{\Crit}{\operatorname{Crit}}

\newcommand{\Spec}{\operatorname{Spec}}
\title{Derivative and multiplicity loci of elementary symmetric polynomials in positive characteristic}
\author{Ying Xie\\Kennesaw State University\\\texttt{yxie2@kennesaw.edu}}
\date{}

\begin{document}
\maketitle
\begin{abstract}
We determine the dimensions of all derivative loci and all higher
multiplicity loci of an elementary symmetric polynomial over an
algebraically closed field of positive characteristic. For the degree-$d$
polynomial in $n$ variables, set $Q=p^{v_p(n-d+1)}$. The common zero
locus of its order-$q$ derivatives has dimension
$d-q-1+\ind_{Q>d-q}$, whereas the locus of zeros of multiplicity at
least $s$ has dimension $d-s+\ind_{Q>d}$. The order-two case settles
the dimension question in a conjecture of Orzel, with a necessary
correction when $d$ is a power of $p$. We classify all components
attaining the larger dimension and prove that they are generically
smooth. Along the coordinate components in the smaller-dimension
cases, we determine the local rings and obtain generic multiplicity
exactly $Q$. Finally, we identify the full translation stabilizer
scheme and derive a Frobenius normal form. These results distinguish
dimension jumps from infinitesimal thickening and give a uniform
description of their dependence on the characteristic.
\end{abstract}


\section{Introduction}

The geometry of derivative loci of elementary symmetric polynomials
has a marked dependence on the characteristic. In characteristic zero,
the critical locus of the degree-$d$ elementary symmetric polynomial
has dimension $d-2$. In positive characteristic, the dimension can
increase by one; moreover, vanishing of the gradient need not imply
vanishing of the polynomial. An exact description must therefore
distinguish the critical scheme from the singular scheme of the
zero hypersurface. It must also account for nonreduced structure,
which can persist even when the dimension does not increase.

This paper determines the dimensions for every derivative order and
every order of vanishing. The answer depends on the single integer
$p^{v_p(n-d+1)}$. The same integer measures the generic thickness of
the coordinate components and determines the translation stabilizer.
Thus the characteristic-dependent exceptions admit both an exact
numerical criterion and a scheme-theoretic explanation.

\subsection{Previous work and the dimension problem}

In characteristic zero, Limaye, Mittal, and Pareek
\cite[Lemma~13]{LMP} prove the upper bound $d-2$ for the critical
locus. Coordinate subspaces give the matching lower bound; see also
the historical discussion in \cite{Orzel}. A general method for bounding
singular loci of symmetric hypersurfaces by partition spaces appears
in Cafure, Matera, and Privitelli \cite[Theorem~3.1 and Remark~3.2]{CMP}.
In particular, their argument bounds the number of distinct coordinate
values and shows that a component attaining the resulting dimension
bound must be an entire partition space.

For elementary symmetric polynomials in positive characteristic,
Orzel \cite{Orzel} proves that the order-two zero locus has dimension
between $d-2$ and $d-1$ and proposes a complete criterion in
Conjecture~2.11. This question arises in the study of formula
complexity and linear projections of elementary symmetric polynomials.
The first-order result below replaces the two possible dimensions by
a necessary and sufficient arithmetic criterion. It also identifies
a prime-power boundary case in which the criterion printed in that
conjecture, and in the associated sufficient condition, must be amended.

The contribution goes beyond this first-order question. We determine
the full numerical profile of higher derivative and multiplicity
loci, classify all components responsible for an increase in
dimension, and compute the local schemes along coordinate subspaces.
The partition-space dimension principle is inherited from earlier
work; the exact congruence criteria and the higher-order and local
scheme calculations are the substance of the results proved here.

\subsection{Notation and dimension formulas}

Let $K$ be an algebraically closed field, and write
\[
 e_d=e_d(x_1,\ldots,x_n)
 =\sum_{\substack{I\subseteq[n]\\|I|=d}}\prod_{i\in I}x_i,
 \qquad 2\le d\le n.
\]
We write $[n]=\{1,\ldots,n\}$ and use the conventions $e_0=1$
and $e_j=0$ when $j$ exceeds the number of variables. For $I\subseteq[n]$, let $\partial_I$ denote
differentiation once in each variable indexed by $I$, with
$\partial_\varnothing e_d=e_d$. Since $e_d$ is multilinear, these
squarefree derivatives agree with the corresponding Hasse derivatives,
and derivatives with a repeated index vanish. Define closed subschemes
of $\A^n_K$ by
\begin{align}
 C_q(e_d)&=V(\partial_Ie_d:|I|=q),
 &&1\le q\le d-1,\label{eq:Cdef}\\
 Z_s(e_d)&=V(\partial_Ie_d:|I|<s),
 &&2\le s\le d.\label{eq:Zdef}
\end{align}
Thus $C_1(e_d)=\Crit(e_d)$, while $Z_2(e_d)$ is defined by $e_d$
and its first derivatives. More generally, the underlying set of
$Z_s(e_d)$ consists of the points at which $e_d$ has order of
vanishing at least $s$. Dimensions refer to underlying closed sets;
local rings and multiplicities refer to the displayed defining ideals.
The symbol $\ind_A$ is $1$ when the condition $A$ holds and $0$ otherwise.

\begin{theorem}[Exact dimensions]\label{thm:dimensions}
Let $K$ be algebraically closed of characteristic $p>0$, and set
\[
 L=n-d+1,\qquad Q=p^{v_p(L)},
\]
where $v_p(L)$ is the exponent of $p$ in the prime factorization of $L$.
Then
\begin{align}
 \dim C_q(e_d)&=d-q-1+\ind_{Q>d-q},&&1\le q\le d-1,\label{eq:Cdim}\\
 \dim Z_s(e_d)&=d-s+\ind_{Q>d},&&2\le s\le d.\label{eq:Zdim}
\end{align}
In characteristic zero the dimensions are $d-q-1$ and $d-s$,
respectively.
\end{theorem}

In positive characteristic we retain the notation $L$ and $Q$ of
Theorem~\ref{thm:dimensions} throughout. The first-order specialization
has the following precise form.

\begin{corollary}[The order-two criterion]\label{cor:criterion}
Let $P_-$ be the smallest $p$-power at least $d$, and let $P_+$ be
the smallest $p$-power strictly greater than $d$. Then
\[
 \dim C_1(e_d)=d-1\ \Longleftrightarrow\ P_-\mid n-d+1,
 \qquad
 \dim Z_2(e_d)=d-1\ \Longleftrightarrow\ P_+\mid n-d+1.
\]
In each case the dimension is $d-2$ when the divisibility fails.
\end{corollary}
\begin{proof}
The inequalities $Q>d-1$ and $Q>d$ are equivalent to the two
divisibility conditions, respectively.
\end{proof}

The two thresholds coincide unless $d$ is a $p$-power. If $d=p^a$
with $a\ge1$, they are $d$ and $pd$. This distinction is exactly
the correction needed in \cite[Conjecture~2.11]{Orzel}; a detailed
comparison and explicit examples are given in Section~\ref{sec:comparison}.

\subsection{Components and local structure}

A partition space is the linear subspace on which coordinates agree
within each block of a set partition of $[n]$. We call a component
\emph{extremal} when its dimension equals the larger of the two values
in the corresponding formula.

\begin{theorem}[Extremal components]\label{thm:components}
In positive characteristic all extremal components are as follows.
\begin{enumerate}
\item For $C_1(e_d)$, partition $[n]$ into $d-1$ blocks, each of size
congruent to $1$ modulo the smallest $p$-power at least $d$.
For $Z_2(e_d)$, use instead the smallest $p$-power strictly greater
than $d$. The associated partition spaces are exactly the extremal
components.
\item For $C_q(e_d)$ with $q\ge2$, put $r=d-q$. When $Q>r$, the
extremal components are the spaces on which all but $r-1$ coordinates
are equal. There are $\binom{n}{r-1}$ such components.
\item For $Z_s(e_d)$ with $s\ge3$ and $Q>d$, the extremal components
are the spaces on which all but $d-s$ coordinates are equal. There
are $\binom{n}{d-s}$ such components.
\end{enumerate}
Every extremal component is generically smooth in the scheme defined
by the relevant derivative ideal, and has multiplicity one.
\end{theorem}

The smaller-dimension cases can have nonreduced coordinate components.
For $S\subseteq[n]$, write $V_S=V(x_j:j\notin S)$ and
$F_S=K(x_i:i\in S)$. Generic multiplicity means the length of the
local ring at the generic point of an irreducible component.

\begin{theorem}[Local rings along coordinate subspaces]\label{thm:local}
Fix $1\le q<d$, put $r=d-q$, and let $\eta_S$ be the generic point
of $V_S$, where $|S|=r-1$. In characteristic $p>0$,
\begin{align*}
 \mathcal O_{C_q(e_d),\eta_S}&\cong
 \begin{cases}F_S[t]/(t^Q),&Q\le r,\\F_S[t]_{(t)},&Q>r,\end{cases}\\
 \mathcal O_{Z_{q+1}(e_d),\eta_S}&\cong
 \begin{cases}F_S[t]/(t^Q),&Q\le d,\\F_S[t]_{(t)},&Q>d.\end{cases}
\end{align*}
In each Artinian case, $V_S$ is a component with generic multiplicity
$Q$. In each other case it lies in a component of dimension $r$.
In characteristic zero both local rings are $F_S$.
\end{theorem}

Theorems~\ref{thm:dimensions}--\ref{thm:local} separate two effects of
positive characteristic. The inequality $Q>d$ increases the dimension
of every $Z_s$ by one. When $1<Q\le d$, these dimensions remain at
their characteristic-zero values, but the coordinate components have
generic multiplicity $Q$. Dimension alone therefore does not detect
the characteristic-dependent scheme structure. The extremal component
classification does not purport to describe all smaller-dimensional
or embedded components.

\subsection{Translation symmetry and the structure of the proof}

Proposition~\ref{prop:stabilizer} identifies the full translation
stabilizer of $e_d$ with its highest multiplicity scheme $Z_d(e_d)$.
It is the diagonal infinitesimal group $\alpha_Q$ when $Q\le d$,
and the diagonal additive group when $Q>d$. In particular, the
dimension increase in every $Z_s$ occurs precisely when the
translation stabilizer is positive-dimensional. A Frobenius normal
form explains why the gradient and the polynomial have different
thresholds when $Q=d$.

The proof proceeds from partition spaces to local rings. A common
monic polynomial bounds the number of distinct coordinate values
at a critical point. On a partition space of the largest possible
dimension, the derivative equations become a truncated generating-
function identity. Frobenius then converts this identity into exact
divisibility conditions on the block sizes. For higher derivatives,
comparing two choices of deleted coordinates forces all but one
block to be singletons. Finally, localizing a difference of two
derivatives identifies the transverse local ring with a quotient
of a one-variable discrete valuation ring; its least nonzero order
is exactly $Q$.

Sections~\ref{sec:first} and \ref{sec:higher} prove the dimension
and component statements. Section~\ref{sec:local} establishes the
local rings and generic smoothness. Section~\ref{sec:translation}
treats translation stabilizers and gradient fibers, and
Section~\ref{sec:comparison} gives the conjecture comparison and examples.

\section{The first derivative and order-two loci}\label{sec:first}

We first give the elementary block argument underlying the dimension
formulas. The bound on the number of distinct coordinates is closely
related to the arguments in \cite[Theorem~3.1]{CMP} and
\cite[Lemma~2.4]{Orzel}.

\begin{lemma}\label{lem:roots}
Every point of $C_1(e_d)$ has at most $d-1$ distinct coordinate values.
Consequently both $C_1(e_d)$ and $Z_2(e_d)$ have dimension at most
$d-1$.
\end{lemma}
\begin{proof}
The deletion identity gives
\[
 \partial_i e_d=\sum_{j=0}^{d-1}(-x_i)^j e_{d-1-j}(x).
\]
At a critical point all its coordinates are roots of the same
polynomial of degree $d-1$, whose leading coefficient is
$(-1)^{d-1}$. There are at most $d-1$ distinct roots. The critical
locus is therefore contained in the finite union of partition
spaces with $d-1$ nonempty blocks. A point with fewer distinct
values also lies in one of these spaces, by splitting blocks.
\end{proof}

\begin{lemma}\label{lem:block}
Put $r=d-1$, and let $D$ be a partition space with block sizes
$m_1,\ldots,m_r$. In characteristic $p>0$, let $P_-$ and $P_+$
be the smallest $p$-powers at least $d$ and strictly greater than
$d$, respectively. Then
\[
 D\subseteq C_1(e_d)\ \Longleftrightarrow\ m_j\equiv1\pmod{P_-}\text{ for all }j,
\]
and
\[
 D\subseteq Z_2(e_d)\ \Longleftrightarrow\ m_j\equiv1\pmod{P_+}\text{ for all }j.
\]
In characteristic zero either containment would require every
$m_j$ to equal $1$.
\end{lemma}
\begin{proof}
Use independent block coordinates $y_1,\ldots,y_r$, and set
\[
 E(T)=\prod_{j=1}^r(1+y_jT)^{m_j},\quad
 R(T)=\prod_{j=1}^r(1+y_jT),\quad
 H(T)=\frac{E(T)}{R(T)}.
\]
If $D\subseteq C_1(e_d)$, the polynomial
\[
 \sum_{j=0}^r(-U)^j e_{r-j}|_D
\]
has the $r$ distinct roots $y_1,\ldots,y_r$ over
$K(y_1,\ldots,y_r)$. It must equal
$(-1)^r\prod_j(U-y_j)$. Comparing coefficients, and conversely
using the same identity, gives
\begin{equation}
 D\subseteq C_1(e_d)\ \Longleftrightarrow\ H(T)\equiv1\pmod{T^d}.
 \label{eq:blockcritical}
\end{equation}
Under this congruence, the coefficient of $T^d$ in $E$ vanishes
exactly when that in $H$ vanishes, because $R$ has degree $d-1$.
Thus
\begin{equation}
 D\subseteq Z_2(e_d)\ \Longleftrightarrow\ H(T)\equiv1\pmod{T^{d+1}}.
 \label{eq:blockzero}
\end{equation}

If a positive integer $M$ has $p$-adic valuation $a$, the first
nonconstant term of $(1+z)^M$ in characteristic $p$ has degree
$p^a$. Indeed, writing $M=p^a b$ with $p\nmid b$ gives
$(1+z)^M=(1+z^{p^a})^b$, with nonzero coefficient $b$ in that
degree. It follows that
\[
 \prod_j(1+y_jT)^{m_j-1}\equiv1\pmod{T^h}
\]
if and only if each $m_j-1$ is divisible by the smallest $p$-power
at least $h$. Necessity follows by setting all but one $y_j$ to
zero, and sufficiency follows by multiplication. This proves the
claims. In characteristic zero the coefficient of $y_jT$ already
forces $m_j-1=0$.
\end{proof}

Both loci contain every coordinate space supported on at most
$d-2$ indices. Their dimensions therefore lie between $d-2$ and
$d-1$. By Lemma~\ref{lem:roots}, an irreducible component of
dimension $d-1$ must be an entire partition space with $d-1$
blocks: a proper closed subset of such a space has smaller dimension.
There exists a partition into $r$ blocks of sizes $1$ modulo $P$
if and only if $P\mid n-r$. For sufficiency, take $r-1$ singleton
blocks and one block of size $1+n-r$. Lemma~\ref{lem:block} now
proves Theorem~\ref{thm:dimensions} for $q=1$ and $s=2$, and the
first part of Theorem~\ref{thm:components}. It also proves the
characteristic-zero dimensions, since $n\ge d>r$ rules out a
partition consisting only of singleton blocks.

\section{Higher derivatives and higher multiplicity}\label{sec:higher}

Fix $2\le q\le d-1$, and put $r=d-q$. A point of $C_q(e_d)$ has
at most $r$ distinct coordinate values. Otherwise choose $r+1$
coordinates with distinct values and $q-1$ further indices $B$
disjoint from them; this is possible since $n\ge r+q$. The
polynomial $\partial_Be_d$ is elementary symmetric of degree
$r+1$ in the remaining variables, and its gradient vanishes at
the point. Lemma~\ref{lem:roots} gives a contradiction.

Thus $C_q$ is contained in the finite union of partition spaces
with $r$ blocks. Both $C_q$ and $Z_{q+1}$ contain the coordinate
spaces supported on $r-1$ indices. Their dimensions lie between
$r-1$ and $r$; an $r$-dimensional component must be a full partition
space of this size.

\begin{lemma}\label{lem:oneblock}
In positive characteristic, let $P_r$ be the smallest $p$-power
strictly greater than $r$. A partition space with $r$ blocks lies
in $C_q(e_d)$ if and only if it has $r-1$ singleton blocks and
one block of size $L+q$, with $P_r\mid L$.
\end{lemma}
\begin{proof}
Suppose the block sizes are $m_1,\ldots,m_r$. Remove $q-1$
indices without emptying any block, say $b_j$ indices from block
$j$. Such a removal is possible because $\sum_j(m_j-1)=n-r\ge q$.
For every such removal, the gradient of the remaining elementary
polynomial of degree $r+1$ vanishes identically on the partition space. By
Lemma~\ref{lem:block},
\begin{equation}
 m_j-b_j\equiv1\pmod{P_r},\qquad
 0\le b_j\le m_j-1,\qquad\sum_jb_j=q-1.\label{eq:removal}
\end{equation}
The capacities $c_j=m_j-1$ sum to $n-r\ge q$. If at least two
are positive, a feasible vector $b$ of total $q-1$ can be changed
by transferring one unit between two distinct coordinates. To see
this, there is a positive entry of $b$ and an unfilled capacity,
because $0<q-1<\sum c_j$. Unless a transfer is possible, both
sets of indices must be the same singleton; all other capacities
would then be zero. The transfer contradicts \eqref{eq:removal}.
Hence exactly one block is nonsingleton. Its size is $M=L+q$.
Removing $q-1$ indices from it shows that $P_r\mid L$.

Conversely, suppose these conditions hold. Differentiate with respect
to $q$ distinct variables, taking $a$ of them from singleton blocks
and $q-a$ from the large block. Denote the common value on the large
block by $b$ and the remaining singleton values by
$y_1,\ldots,y_{r-1-a}$. The restricted derivative is the coefficient
of $T^r$ in
\[
 \prod_{j=1}^{r-1-a}(1+y_jT)(1+bT)^{L+a}.
\]
Modulo $T^{r+1}$, replace the last exponent by $a$, since
$P_r>r$ and $P_r\mid L$. The resulting polynomial has degree
$(r-1-a)+a=r-1$. Its coefficient of $T^r$ is zero.
\end{proof}

In characteristic zero the same removal argument would require
every remaining block size to be $1$. Their sum is
$n-q+1\ge r+1$, which is impossible. Lemma~\ref{lem:oneblock}
therefore proves \eqref{eq:Cdim}, the characteristic-zero case,
and the component description for $C_q$.

For $Z_{q+1}$, let $P$ be the smallest $p$-power strictly greater
than $d$. An extremal partition already has the shape in
Lemma~\ref{lem:oneblock} and satisfies $P_r\mid L$. It remains to
show that $P\mid L$. If $P_r>d$, then $P_r=P$: the inequality
$r<d$ gives $P_r\le P$, and $P_r>d$ gives $P\le P_r$ by the
minimality of $P$. Thus $P\mid L$ already holds in this case.

Now suppose $P_r\le d$. Let $R$ be a $p$-power with $r<R\le d$
and $R\mid L$. Differentiate with respect to $h=d-R$ variables,
all in the large block. This is an allowed order since $0\le h<q$.
The coefficient of the pure monomial $b^R$ in the restricted
derivative is
\[
 \binom{L+R-r}{R}\equiv L/R\pmod p.
\]
Here $0<R-r<R$; the congruence follows by expanding
$(1+z)^{L+R-r}=(1+z^R)^{L/R}(1+z)^{R-r}$ in characteristic $p$.
Vanishing forces $pR\mid L$. Starting at $R=P_r$ and applying
this step while $R\le d$ proves divisibility by the next $p$-power
at each stage. The first $p$-power greater than $d$ is $P$, so the
iteration gives $P\mid L$.

Conversely, assume $P\mid L$. For a derivative of order
$0\le h\le q$, let $a$ be the number of differentiated variables
from singleton blocks. Write the remaining singleton values as
$y_1,\ldots,y_{r-1-a}$. Its generating polynomial is
\[
 \prod_{j=1}^{r-1-a}(1+y_jT)
 (1+bT)^{L+q-h+a}.
\]
Modulo $T^P$, replace the last exponent by $q-h+a$. The resulting
degree is at most $d-h-1$, while the required coefficient has
degree $d-h<P$. It vanishes. This proves \eqref{eq:Zdim} and
the remaining component descriptions. The number of these partition
spaces is the number of choices of singleton coordinates, as stated.

\section{Local multiplicities and generic smoothness}\label{sec:local}

\subsection{Coordinate components}

We prove Theorem~\ref{thm:local}. Put $B=[n]\setminus S$, so
$|B|=M=L+q$. Work in the local ring of $\A^n$ at
$(x_j:j\in B)$. For $i,j\in B$, choose $q-1$ other indices
$T\subset B\setminus\{i,j\}$. The identity
\begin{equation}
 \partial_{T\cup\{i\}}e_d-\partial_{T\cup\{j\}}e_d
 =(x_j-x_i)\partial_{T\cup\{i,j\}}e_d\label{eq:difference}
\end{equation}
has a unit as its last factor: $|B|=L+q\ge q+1$ guarantees the
choice of $T$, and the restriction of that factor to $V_S$ is
$\prod_{a\in S}x_a$. Thus both derivative ideals contain all
differences $x_i-x_j$ on $B$ after localization. The ambient local
ring is $F_S[x_j:j\in B]_{(x_j:j\in B)}$. Quotienting by these
differences and writing their common value as $t$ gives
$F_S[t]_{(t)}$. It remains to compute the ideal generated by the
restricted derivatives in this discrete valuation ring.

Consider a derivative of order $h\le q$ removing $a$ indices
from $S$, and put $u=q-h$. Its restriction is the coefficient
of $T^{d-h}$ in
\begin{equation}
 \prod_{i\in S\setminus A}(1+x_iT)(1+tT)^{L+u+a},
 \qquad |A|=a.\label{eq:localgen}
\end{equation}
A possible power $t^j$ satisfies $u+a+1\le j\le r+u$.
If $j<Q$, then $u+a<Q$, and its binomial coefficient reduces to
$\binom{u+a}{j}=0$ modulo $p$. Hence every nonzero restricted
generator has $t$-order at least $Q$.

For the order-$q$ ideal, the derivative using only indices in $B$
restricts to
\[
 \sum_{j=1}^r\binom Lj t^j e_{r-j}(x_S).
\]
When $Q\le r$, its coefficient of $t^Q$ is the nonzero element
$(L/Q)e_{r-Q}(x_S)$ of $F_S$. The ideal in the discrete valuation
ring is therefore $(t^Q)$. If $Q>r$, all order-$q$ restrictions
vanish identically by the block criterion.

For the ideal of $Z_{q+1}$, this also suffices when $Q\le r$.
If $r<Q\le d$, use the derivative of order $h=d-Q$, all in $B$.
In \eqref{eq:localgen} the coefficient of $t^Q$ is
$\binom{L+Q-r}{Q}=L/Q\ne0$ in $K$. If $Q>d$, all restrictions
vanish identically, by the preceding section. The dimension
formulas identify the component cases, and the length of
$F_S[t]/(t^Q)$ is $Q$. In characteristic zero the coefficient of
$t$ in the order-$q$ restriction is
$L\prod_{a\in S}x_a\ne0$, giving the local ring $F_S$.

\subsection{Extremal components}

For $q=1$, take a partition component with distinct block values
$y_1,\ldots,y_{d-1}$. Extracting the coefficient of degree $d-2$
after two deletions shows that the Hessian of $e_d$ has zero
cross-block entries, and block $j$ equals
\[
 c_j(J_{m_j}-I_{m_j}),\qquad
 c_j=\prod_{\ell\ne j}(y_\ell-y_j)\ne0.
\]
Here $J_m$ is the all-one matrix and $I_m$ the identity matrix.
Since $m_j\equiv1\pmod p$, the kernel of $J_{m_j}-I_{m_j}$ is
exactly the line of constant vectors, so its rank is $m_j-1$. The Hessian
rank is therefore $n-(d-1)$, exactly the codimension of the
partition space. The critical scheme is smooth there. If that
space also lies in $Z_2$, the latter scheme has the same dimension
and is smooth there as well.

For $q\ge2$, put $r=d-q$ and take an extremal partition with
$r-1$ singleton values $a_i$ and a common value $b$ on the block
$B$ of size $M=L+q$. Assume $a_i\ne b$ for all $i$. In the
Jacobian of the order-$q$ equations, use the rows indexed by
$q$-subsets $I$ of $B$. Their entries at singleton coordinates
are zero; the entry at $j\in B$ is
\[
 \begin{cases}c,&j\notin I,\\0,&j\in I,\end{cases}
 \qquad c=\prod_i(a_i-b)\ne0.
\]
The complement-incidence rows have coordinate sum $M-q=L=0$
in $K$. For distinct $i,j\in B$, choose a $(q-1)$-subset of
$B\setminus\{i,j\}$ and compare the two $q$-subsets obtained by
adjoining $i$ and $j$. The corresponding row difference is a
nonzero scalar multiple of $e_i-e_j$. Thus the row space contains,
and is contained in, the coordinate-sum-zero subspace on $B$.
Its rank is $M-1=n-r$. The scheme contains the
$r$-dimensional partition space, so this is precisely the rank
required for smoothness. The same argument applies to $Z_{q+1}$.
This completes the proof of Theorem~\ref{thm:components}.

\section{Translation stabilizers and a normal form}\label{sec:translation}

For each $K$-algebra $A$, consider the translation stabilizer
\[
 \mathcal T(A)=\{v\in A^n:e_d(x+v)=e_d(x)\text{ in }A[x]\}.
\]
Write $\alpha_Q=\Spec K[t]/(t^Q)$ when $Q$ is a $p$-power;
this is the kernel of the $Q$th-power homomorphism of $\Ga$.

\begin{proposition}\label{prop:stabilizer}
In positive characteristic the translation stabilizer is the
diagonal copy of $\alpha_Q$ if $Q\le d$, and the diagonal copy
of $\Ga$ if $Q>d$. In characteristic zero it is the reduced
origin. Moreover, $\mathcal T=Z_d(e_d)$ as closed subschemes.
\end{proposition}
\begin{proof}
The coefficient of $x_I$, for $|I|=d-1$, in
$e_d(x+v)-e_d(x)$ is $\sum_{j\notin I}v_j$. Subtracting the
equations for two such subsets differing by one index forces
$v_i=v_j$ for every $i,j$, over any test algebra. Write the common
coordinate as $t$. The remaining equations are
\begin{equation}
 \binom{L+h-1}{h}t^h=0,\qquad1\le h\le d.\label{eq:stabeq}
\end{equation}
If $L=Qu$ with $p\nmid u$, the generating series is
\[
 (1-z)^{-L}=(1-z^Q)^{-u}.
\]
Its first nonconstant coefficient has degree $Q$ and equals the
unit $u$; all its positive degrees are multiples of $Q$. Thus
\eqref{eq:stabeq} generates $(t^Q)$ when $Q\le d$, and the zero
ideal when $Q>d$. In characteristic zero the first equation
already gives $Lt=0$.

For a homogeneous multilinear polynomial of degree $d$, the
coefficients of degrees less than $d$ in its Taylor expansion
at $v$ are exactly the derivatives of orders less than $d$ at
$v$. Its degree-$d$ part is the original polynomial. Therefore
the stabilizer equations define precisely $Z_d(e_d)$.
\end{proof}

In particular, the defining ideal of $Z_d$ in positive characteristic
is
\[
 \begin{cases}
 (x_1-x_n,\ldots,x_{n-1}-x_n,x_n^Q),&Q\le d,\\
 (x_1-x_n,\ldots,x_{n-1}-x_n),&Q>d.
 \end{cases}
\]
The dimension jump in every $Z_s$ thus occurs exactly when the
translation stabilizer has positive dimension. A nontrivial finite
infinitesimal stabilizer alone does not produce this jump.

There is also an explicit normal form. Set $z_i=x_i-x_n$ for
$i<n$ and $t=x_n$. The translation expansion gives
\begin{equation}
 e_d(x)=\sum_{b=0}^{\lfloor d/Q\rfloor}
 \binom{u+b-1}{b}t^{bQ}e_{d-bQ}(z_1,\ldots,z_{n-1}),
 \qquad u=L/Q.\label{eq:normal}
\end{equation}
Here $e_0=1$ and $e_j=0$ for $j>n-1$. When $Q>d$, the
polynomial is independent of $t$. When $Q=d$, it is
$e_d(z)+ut^d$; its gradient is still independent of $t$.
Thus the critical-locus threshold is $Q\ge d$, whereas the
order-two zero threshold is $Q>d$. These differ precisely at
$p$-power degrees.

\begin{corollary}\label{cor:gradientfibers}
If $Q\ge d$, every gradient fiber over a target
$a=(a_1,\ldots,a_n)$ with $\sum_i a_i=0$ is an affine line
times the corresponding gradient fiber of $e_d$ in $n-1$
variables. Targets outside that hyperplane have empty fibers.
\end{corollary}
\begin{proof}
Let $g=\nabla e_d(z_1,\ldots,z_{n-1})$. Differentiating
\eqref{eq:normal}, in the original coordinates, gives
\[
 (g_1,\ldots,g_{n-1},-\textstyle\sum_i g_i),
\]
independently of $t$.
This is an identity of morphisms and hence gives the stated
scheme-theoretic product of fibers.
\end{proof}

\section{The order-two criterion and examples}\label{sec:comparison}

Let $P_+$ be the smallest $p$-power strictly greater than $d$.
Formula \eqref{eq:Zdim} says
\[
 \dim Z_2(e_d)=d-1\quad\Longleftrightarrow\quad
 n-d+1\equiv0\pmod{P_+}.
\]
Conjecture~2.11 of \cite{Orzel} uses instead
$p^{\lceil\log_p d\rceil}$, together with $n\ge2d-1$.
Proposition~2.6 of that revision states a sufficient condition
with the same threshold. When $d$ is not a $p$-power, that threshold is
$P_+$, and its divisibility condition already implies $n\ge2d$.
When $d$ is a $p$-power, the strict threshold must instead be
$pd$. Theorem~\ref{thm:dimensions} proves the necessary and
sufficient criterion with this correction.

For the smallest example, take $p=2,d=2,n=3$. The gradient of
\[
 e_2=x_1x_2+x_1x_3+x_2x_3
\]
vanishes on the diagonal $(a,a,a)$, where $e_2=a^2$. Hence
$\dim C_1(e_2)=1$ but $Z_2(e_2)$ is supported at the origin.
Its coordinate ring is $K[t]/(t^2)$ by
Proposition~\ref{prop:stabilizer}. The printed conditions
$n\ge2d-1$ and $p^{\lceil\log_p d\rceil}\mid n-d+1$ hold,
but $\dim Z_2(e_2)=0$.

An odd-characteristic example is $p=3,d=3,n=5$. On the partition
space $(a,b,b,b,b)$, every first derivative vanishes, while
$e_3=b^3$. The exact dimensions are $\dim C_1(e_3)=2$ and
$\dim Z_2(e_3)=1$. Every coordinate-axis component of the latter
has generic multiplicity three. More generally, when $d=p^a$
and $n=2d-1$, the two dimensions are $d-1$ and $d-2$.

The correction concerns the arithmetic criterion for attaining
dimension $d-1$. It leaves unchanged the uniform upper bound used
in the formula lower-bound argument of \cite{Orzel}.

\section{Consequences and further questions}

The results determine the full dimension profile of the schemes
$C_q(e_d)$ and $Z_s(e_d)$ and explain the exceptional dimensions
through translation symmetry. Their scheme-theoretic content is
independent of a dimension increase: if $1<Q\le d$, the coordinate
components of $Z_s$ have nontrivial generic thickness even though
their dimensions agree with those in characteristic zero. At the
other extreme, $Q>d$ gives an actual translation direction and
increases every higher multiplicity dimension by one.

For the first-order problem, the distinction is especially explicit.
At $Q=d$, the normal form is $e_d(z)+ut^d$. The term $ut^d$ is
invisible to ordinary first derivatives, but it remains in the
equation of the zero hypersurface. This accounts simultaneously
for the difference between $C_1$ and $Z_2$, the strict divisibility
threshold in Corollary~\ref{cor:criterion}, and the infinitesimal
stabilizer $\alpha_d$.

In applications that use derivative-locus codimension, the formulas
supply the exact geometric input in every characteristic and degree.
The present paper does not derive a new asymptotic formula or circuit
lower bound from that refinement. Such an application would require
an additional argument connecting the relevant computational model
to these exact loci.

Several geometric questions remain. The smaller-dimensional and
embedded components are not classified here. Determining primary
decompositions or a general criterion for global reducedness would
require information beyond the generic coordinate local rings.
Likewise, Corollary~\ref{cor:gradientfibers} describes all gradient
fibers in the range where the gradient is invariant under diagonal
translations, but does not classify the fibers in the remaining
small-characteristic cases. These
questions concern the geometry left after the numerical dimension
problem has been resolved.

\section*{Acknowledgment}
The author acknowledges OpenAI GPT-6 for assistance with research exploration,
proof development, verification code, and manuscript preparation.

\end{document}